\documentclass[12pt]{amsart}
\pdfoutput=1

\usepackage{amssymb, amsmath, amsthm, mathrsfs}

\usepackage{titletoc}
\usepackage[shortlabels]{enumitem}

\usepackage{amsfonts}

\usepackage[utf8]{inputenc}

\usepackage{latexsym}

\usepackage[usenames,dvipsnames]{xcolor}
\definecolor{darkblue}{rgb}{0.0, 0.0, 0.55}

\usepackage[pagebackref,colorlinks,linkcolor=BrickRed,citecolor=darkblue,urlcolor=black,hypertexnames=true]{hyperref}

\usepackage{bbm}

\newtheorem{theorem}{Theorem}[section]
\newtheorem{cor}[theorem]{Corollary}

\newtheorem{lemma}[theorem]{Lemma}

\newtheorem{prop}[theorem]{Proposition}
\newtheorem{example}[theorem]{Example}
\theoremstyle{definition}

\newtheorem{rem}[theorem]{Remark}

\DeclareFontFamily{U}{mathx}{\hyphenchar\font45}
\DeclareFontShape{U}{mathx}{m}{n}{
      <5> <6> <7> <8> <9> <10>
      <10.95> <12> <14.4> <17.28> <20.74> <24.88>
      mathx10
      }{}
\DeclareSymbolFont{mathx}{U}{mathx}{m}{n}
\DeclareFontSubstitution{U}{mathx}{m}{n}
\DeclareMathAccent{\widecheck}{0}{mathx}{"71}
\DeclareMathAccent{\wideparen}{0}{mathx}{"75}

\def\beq{\begin{equation}
}
\def\eeq{\end{equation}}

\makeatletter
\def\moverlay{\mathpalette\mov@rlay}
\def\mov@rlay#1#2{\leavevmode\vtop{
		\baselineskip\z@skip \lineskiplimit-\maxdimen
		\ialign{\hfil$#1##$\hfil\cr#2\crcr}}}
\makeatother

\newcommand{\plangle}{\moverlay{(\cr<}}
\newcommand{\prangle}{\moverlay{)\cr>}}
\newcommand{\freefstar}{\C\plangle x,x^* \prangle}
\newcommand{\freef}{\C\plangle x \prangle}

\newcommand{\lax}{\langle x,x^*\rangle}
\newcommand{\laxh}{\langle x,x^*,h,h^*\rangle}

\newcommand{\AB}{B}
\newcommand{\JK}{K}

\def\tX{{\widetilde{X}}}

\def\wtX{{\widetilde X}}

\newcommand{\wtH}{\widetilde{H}}

\newcommand{\cJ}{\mathcal{J}}
 \def\tcJ{{\widetilde \cJ}}

\def\Del{{ \Delta}}
\def\tDel{\widetilde \Delta}

\def\beq{\begin{equation}}
\def\eeq{\end{equation}}

\def\ben{\begin{enumerate}}
\def\een{\end{enumerate}}

\def\bem{\begin{bmatrix}}
\def\eem{\end{bmatrix}}

\def\cN{ {{\mathcal N}}}
\def\cP{ {{\mathcal P}}}

\def\bbB{ {\mathbb B}}

\renewcommand{\subset}{\subseteq}

\def\cN{ {\mathcal N} }

\def\cP{{\mathcal P}}

\def\cR{{ \mathcal R }}

\def\cS{{\mathcal S} }

\def\beq{\begin{equation}}
\def\eeq{\end{equation}}

\def\tcJ{\widetilde{\cJ}}

\def\cR{{\mathcal R}}

\def\Lam{{\Lambda}}

\newcommand{\df}[1]{{\bf{#1}}{\index{#1}}}

\newcommand{\mat}[1]{M_{#1}}

\def\spann{\operatorname{span}}

\def\dim{\operatorname{dim}}

\def\tcN{\widetilde{\cN}}

\newcommand{\Pdown}{P}
\newcommand{\Pup}{P_{\ast}}

\newcommand{\EE}{E}
\newcommand{\ao}{\rho}
\newcommand{\SA}{{\mathfrak{S}^-}}
\newcommand{\bx}{\mathbbm{x}}
\newcommand{\by}{\mathbbm{y}}
\newcommand{\bw}{\mathbbm{w}}

\def\CC{{\mathbb C}}

 \def\Del{{\Delta}}

\def\j1tog{ j= 1, \ldots, g }

\DeclareMathOperator{\ran}{rng}
\DeclareMathOperator{\rng}{rng}

\newcommand{\C}{\mathbb{C}}
\newcommand{\ve}{\varepsilon}
\newcommand{\Langle}{\mathop{<}\!}
\newcommand{\Rangle}{\!\mathop{>}}

\newcommand{\tOm}{\widetilde \Omega}
\newcommand{\tGm}{\widetilde{\Gamma}}

\newcommand{\YY}{\widehat{X}}
\newcommand{\KK}{\widehat{H}}

\newcommand{\ax}{\langle x \rangle}
\newcommand{\axplus}{\ax_+}
\newcommand{\axy}{\langle x,x^*\rangle}
\newcommand{\axyplus}{\axy_+}

\newcommand{\Bp}{\widehat{B}}
\newcommand{\rp}{\widehat{r}}
\newcommand{\qp}{\widehat{q}}

\newcommand{\ap}{a}
\newcommand{\bq}{b}

\newcommand{\gv}{{\tt{g}}}
\newcommand{\kv}{{\tt{k}}}
\newcommand{\hv}{{\tt{h}}}

\title[Plush noncommutative rational functions]{Plurisubharmonic noncommutative rational functions}

\author[H. Dym]{Harry Dym}
\address{Harry Dym, Department of Mathematics, Weizmann Institute of Science, Rehovot, Israel}
\email{Harry.dym@weizmann.ac.il}

\author[J.W. Helton]{J. William Helton${}^1$}
\address{J. William Helton, Department of Mathematics\\
  University of California \\
  San Diego}
\email{helton@math.ucsd.edu}
\thanks{${}^1$Research supported by the NSF grant
DMS 1500835.}

\author[I. Klep]{Igor Klep${}^{2}$}
\address{Igor Klep, Department of Mathematics, 
University of Ljubljana, Slovenia}
\email{igor.klep@fmf.uni-lj.si}
\thanks{${}^2$Supported by the 
Slovenian Research Agency grants J1-8132, N1-0057 and P1-0222. Partially supported 
by the 
Marsden Fund Council of the Royal Society of New Zealand.}

\author[S. McCullough]{Scott McCullough${}^3$}
\address{Scott McCullough, Department of Mathematics\\
  University of Florida\\ Gainesville 
   }
   \email{sam@math.ufl.edu}
\thanks{${}^3$Research supported by  NSF grant  DMS-1764231}

\author[J. Vol\v{c}i\v{c}]{Jurij Vol\v{c}i\v{c}${}^4$}
\address{Jurij Vol\v{c}i\v{c}, Department of Mathematics, Texas A\&M University}
\email{volcic@math.tamu.edu}
\thanks{${}^4$Research supported by the Deutsche Forschungsgemeinschaft (DFG) Grant No. SCHW 1723/1-1.}


\makeindex
\makeatletter
\newcommand{\mycontentsbox}{%
{
\parskip=-1.0pt
\newpage\printindex\tableofcontents}}
\def\enddoc@text{\ifx\@empty\@translators \else\@settranslators\fi
\ifx\@empty\addresses \else\@setaddresses\fi
\newpage\mycontentsbox
}
\makeatother

\contentsmargin{2.55em} 

 \numberwithin{equation}{section}

\begin{document}

\allowdisplaybreaks

\begin{abstract}
A noncommutative (nc) function 
in $x_1,\dots,x_\gv,x_1^*,\dots,x_\gv^*$ is called
plurisubharmonic (plush) 
if its nc complex Hessian takes only positive semidefinite values on 
an nc neighborhood of $0.$
The main result of this paper shows
that an nc rational function  is plush
if and only if it is a composite of
a convex rational function with an analytic 
(no $x_j^*$) 
rational function. The  proof is entirely constructive.
Further,  a simple computable
necessary and sufficient
condition for an nc rational function   to be
plush is given
in terms of its minimal realization.\looseness=-1
\end{abstract}

\maketitle

 \dottedcontents{section}[3.8em]{}{2.3em}{.4pc} 
 \dottedcontents{subsection}[6.1em]{}{3.2em}{.4pc}
 \dottedcontents{subsubsection}[8.4em]{}{4.1em}{.4pc}

\thispagestyle{empty}

\section{Introduction} \label{sec:Intro}
This article establishes a representation theorem (Theorem \ref{t:main-intro}) for free noncommutative (nc) plurisubharmonic rational functions and an effective criterion (Theorem \ref{t:isplush-intro}) for an nc rational function to be plurisubharmonic.
 Plurisubharmonic functions are multivariate analogs
of subharmonic functions and are central objects
 in several complex variables \cite{DAn93,For17}, in part because of their connection
 to pseudoconvex domains.  Our interest in nc plurisubharmonic 
 rational functions  stems from their connection to free domains that 
 can be transformed, via a proper nc rational mapping, to a convex free domain.
 Free domains and free maps are basic objects studied in free analysis \cite{AM15,BMV,MT16,MS,PT-D17,Po1,Po2,SSS18}, a quantized analog of classical analysis. 

\subsection{Basic notation and terminology}
Let $\lax$ denote the free monoid generated by the $2\gv$ freely noncommuting variables 
$x_1, \ldots, x_\gv, x_1^*, \ldots, x_\gv^*$.  \index{$\lax$} Elements of $\lax$
 are \df{words}.   There is a natural involution ${}^*$ on
 $\lax$ determined by $x_j\mapsto x_j^*$ and, $(uv)^* = v^* u^*$ for words $u,v\in \lax$.
Let $\C\lax$ denote the free algebra of finite $\C$-linear combinations of elements
 of $\lax$.  Elements of $\C\lax$ are \df{(nc) polynomials}.    Thus an nc polynomial $p$ has the form
\begin{equation}
\label{e:genericp}
  p =\sum_{w\in \lax}  p_w \, w,
\end{equation}
where the sum is finite and $p_w\in \C.$
The involution  ${}^*$ extends to an involution on $\C\lax$. For $p$ of the form  \eqref{e:genericp},
\[
 p^* = \sum \overline{p_w} \, w^*.
\]
A  polynomial $p\in \C\lax$ is symmetric if $p^*=p$ and is 
\df{analytic} if it contains only the $x$ 
 variables and none of the $x^*$ variables.  In this latter case we write
 $p(x)\in \C\ax$ instead of $p(x,x^*)\in\C\lax.$

Differentiation of elements of $\C\lax$ is described as follows. 
Let $h_1,\dots,h_\gv,h_1^*,\dots,h_\gv^*$ denote a second $2\gv$-tuple of freely
noncommuting variables. For $p\in \C\lax$,  the \df{ partial of $p$  with respect to $x$} 
 and the \df{partial of $p$ with respect to $x^*$} are, respectively,
\[
\begin{split}
 p_x(x,x^*)[h,h^*] = &  \lim_{t\to 0} \frac{p(x+th,x^*) -p(x,x^*)}{t}, \\
 p_{x^*}(x,x^*)[h,h^*] = & \lim_{t\to 0} \frac{p(x,(x+th)^*)-p(x,x^*)}{t}.
\end{split}
\]
  There are four second order  partial derivatives. Each lies in
 $\C\laxh.$ For instance,
\[
 p_{x,x^*}(x,x^*)[h,h^*] 
  = \lim_{t\to 0} \frac{p_{x}(x,(x+th)^*)[h,h^*]-p_{x}(x,x^*)[h,h^*]}{t}
\]
is the \df{complex Hessian} of $p$.

\begin{example}\rm
 Consider the polynomial $q(x,x^*) = 1 + 2x_1 x_2^*x_1^*x_2.$
 Its derivative with respect to $x$ is,
\[
q_x(x,x^*)[h,h^*] = 2h_1 x_2^* x_1^* x_2 + 2x_1x_2^*x_1^* h_2 \in \C\laxh.
\]
 and its complex Hessian is,
\[
\pushQED{\qed} 
 q_{x,x^*}(x,x^*)[h,h^*] = 2h_1 h_2^* x_1^* x_2 + 2h_1 x_2^* h_1^* x_2
  + 2x_1 h_2^* x_1^* h_2 + 2 x_1 x_2^* h_1^* h_2. \qedhere
  \popQED
\]

\end{example}

\begin{example}\label{ex:sumo}\rm 
 As a general example, given analytic polynomials $f_j(x)$, the complex Hessian of 
\begin{equation*}
 Q(x,x^*) = \sum_j\zeta_j^*(x) \zeta_j(x)
\end{equation*}
is 
\[
\pushQED{\qed} 
 Q_{x,x^*}(x,x^*)[h,h^*] = \sum_j \zeta_{x}^*(x)[h] \zeta_x(x)[h].     
 \qedhere
  \popQED
\]
 \end{example}

Let $\mat{n}(\C)^\gv$ denote the set of $\gv$-tuples $X=(X_1,\dots,X_\gv)$ of $n\times n$
 matrices over $\C$. Let $M(\C)^\gv$ denote the sequence $(\mat{n}(\C)^\gv)_n$.
An element $p$ of $\C\lax$ is naturally evaluated at a tuple $X\in M(\C)^\gv$
by simply replacing $x_j$ by $X_j$ and $x_j^*$ by $X_j^*$.  
The involution on $\C\lax$ and evaluation on $M(\C)^g$ is compatible with matrix adjoint; that is,\looseness=-1
$$
p^*(X,X^*) = p(X,X^*)^*.
$$
Moreover, it is well known  and easy to see 
  that $p$ is symmetric if and only if $p(X,X^*)^*=p(X,X^*)$
 for all $X\in M(\C)^\gv$.  

The derivatives of $p$ involve both $x$ and $h$ variables and are thus
evaluated at pairs $(X,H)\in M(\C)^{2\gv}.$   Moreover, 
the derivatives of $p$ are compatible with differentiation after evaluation. 
For example,
\[
 p_x(X,X^*)[H,H^*] = \lim_{t\to 0} \frac{p(X+tH,X^*)-p(X,X^*)}{t}.
\]

A polynomial $p\in \C\lax$ is \df{(matrix) positive} if $p(X,X^*)\succeq 0$
for all $X\in M(\C)^\gv$. Here $T\succeq 0$ indicates the self-adjoint matrix 
$T$ is positive semidefinite.  For example, 
 for the polynomial $Q$ of Example \ref{ex:sumo},
\[
 Q_{x,x^*}(X,X^*)[H,H^*] = \sum_j (\zeta_x(X)[H])^* \zeta_x(X)[H] \succeq 0.
\]
Thus $Q_{x,x^*}$ is matrix positive.

A polynomial $p\in \C\lax$ is \df{plurisubharmonic}, abbreviated \df{plush},
if its complex Hessian is matrix positive. By the main result of \cite{Greene}
(see also \cite{GHVpreprint}), if $p\in \C\lax$ is plush, 
 then $p$  has the (canonical) form,
\begin{equation}
\label{e:conform}
 p(x,x^*) = \ell(x)+\ell(x)^*  +\sum_{j=1}^N \zeta_j(x)^*\zeta_j(x) + \sum_{k=1}^M  \eta_k(x) \eta_k(x)^*,
\end{equation}
for some affine linear analytic $\ell$  and analytic $\zeta_j,\eta_k\in \C\ax$.

A symmetric polynomial $f\in \C\lax$ is \df{convex} if 
\[
  F\left(\frac{X+Y}{2}\right) -\frac{F(X)+F(Y)}{2}\succeq 0
\]
for all $X,Y,$ where $F(X)=f(X,X^*).$
 Convexity of $f$ is equivalent to its \df{(full) Hessian}, defined as
\[
 f^{\prime\prime}(x,x^*)[h,h^*] := f_{x,x}(x,x^*)[h,h^*]  +2f_{x,x^*}(x,x^*)[h,h^*]
 + f_{x^*,x^*}(x,x^*)[h,h^*],
\]
being matrix positive \cite[Theorem 2.4]{HM04}.
Furthermore, by \cite[Theorem 3.1]{HM04},
$f$ is convex if and only if 
 there exists an affine linear analytic polynomial $\ell\in\C\ax$ and linear polynomials
 $\varphi_j\in \C\lax$ such that 
\[
 f(x,x^*) = \ell(x)+\ell(x)^* + \sum_j \varphi_j(x,x^*)^* \varphi_j(x,x^*). 
\]
Hence, writing $\varphi_j(x,x^*)=w_j(x)+y_j(x)^*$, 
 if $f$ is convex, then there exists an analytic (quadratic) polynomial
$u(x)$, a positive integer $M$,
 and  linear analytic polynomials $w_j$ and $v_j$ such that
\begin{equation*}
 f(x,x^*) = u(x)+u(x)^* + \sum_{j=1}^M w_j(x)^*w_j(x) + \sum_{j=1}^M y_j(x) y_j(x)^*.
\end{equation*}

 There is an intimate connection between convex and plush polynomials. 
 Using variables $z=(u,w_1,\dots, w_N, y_1,\dots,y_M)$ and the formal adjoints
 $z^*=(u^*,w_1^*,\dots,w_N^*,y_1^*,\dots,y_N^*)$, the discussion above shows
\begin{equation}
\label{e:fzzstar}
 f(z,z^*) = u+u^* + \sum_{j=1}^N w_j^* w_j + \sum_{j=1}^M y_j y_j^*
\end{equation}
 is convex. Further, for the polynomial $p$ of equation \eqref{e:conform}
 and $f$ from equation \eqref{e:fzzstar},
\[
 p(x,x^*) = f(q(x),q(x)^*),
\]
where $q$ is the analytic mapping,
\[
q(x) = \begin{bmatrix} \ell(x), \ \zeta_1(x), \,  \ldots, \, \zeta_N(x), \,  
   \eta_1(x), \,  \ldots, \,  \eta_M(x) \end{bmatrix}.
\]
Thus, if $p$ is plush,  then $p$ is the composition of
an analytic polynomial with a convex polynomial. The converse is evidently true.
The main result of this paper establishes the analog of this result for nc rational functions.\looseness=-1

\subsection{Noncommutative rational functions}
A {\bf descriptor realization} \cite{BGM,HMV06,KVV} of an nc rational function 
 $r\in\freefstar$ \cite{BR,Coh} regular at $0$  is an expression of
the form
\begin{equation}
\label{e:ratstate0}
r(x,x^*) = c^*( J -  \Lam_A(x)  -  \Lam_{B}(x^*) )^{-1} b,
\end{equation}
where, for some positive integer $d$, the $d\times d$ matrix  $J$ is invertible, $b,c\in \C^d$, $A,B\in M_d(\C)^\gv$ and
\[
 \Lambda_A(x) =\sum_{j=1}^\gv  A_j x_j.
\]
The $d\times d$ matrix-valued polynomial $\Lambda_A(x)\in M_d(\C\ax)$ is evaluated at a $\gv$-tuple
 $X\in M(\C)^\gv$ via the tensor product. Thus if $X\in \mat{n}(\C)^\gv$, then
\[
 \Lambda_A(X) = \sum_{j=1}^\gv A_j\otimes X_j \in M_d(\C)\otimes \mat{n}(\C).
\]
The descriptor realization of equation \eqref{e:ratstate0} is naturally evaluated 
at any tuple $X\in \mat{n}(\C)^\gv$ for which $J-\Lambda_A(X) -\Lambda_B(X^*)$ is invertible as
\[
 r(X,X^*) = (c^*\otimes I_n) \, \left (J\otimes I_n -\Lambda_A(X) -\Lambda_B(X^*)\right)^{-1} \,(b\otimes I_n).
\]
 In particular, $0\in \C^\gv$ is in the domain of $r,$ a property we glorify by
 saying $r$ is \df{regular at $0$}.\looseness=-1

If $r$ from \eqref{e:ratstate0} is symmetric in that 
$r=r^*$, then it admits a \df{symmetric descriptor realization}
\begin{equation}
\label{e:ratstate}
r(x,x^*) = c^*( \JK -  \Lam_\AB(x)  -  \Lam_{\AB^*}(x^*) )^{-1} c,
\end{equation}
where the $d\times d$ matrix  $\JK$ is a \df{signature matrix}
($\JK^2=I_d,\ \JK^*=\JK$).
If $r$ from \eqref{e:ratstate0} is \df{analytic}, i.e.,
has no $x^*$ variables, then we may take $B=0$ in which case
\begin{equation*}
r(x) = c^*( J-  \Lam_{A}(x)   )^{-1} b.
\end{equation*}

 For the purposes of this article, nc rational functions that 
 are regular at $0$ can be identified with
any one of their descriptor realizations as explained in further detail
 in Subsection \ref{sec:realizations}.\looseness=-1

The definitions of derivatives for polynomials naturally extend to symmetric and 
analytic rational functions. Formulas for the derivative, Hessian and 
 complex Hessian of a symmetric descriptor realization are given in Subsection \ref{sec:realizations}.
 In particular, a (symmetric) rational function $r$
 is defined to be plush in a neighborhood of $0$ if its complex Hessian is matrix positive
 in a neighborhood of $0.$ Likewise the notion of convexity for nc polynomials
 extends to nc rational functions.

\subsection{Main results}
We now state the main results of this article.

\begin{theorem}
 \label{t:main-intro}
   A symmetric nc  rational function $r$ 
in $\gv$ variables  that is
   regular at $0$ 
   is plush in a neighborhood of $0$ if and only if
    there exists a positive integer $\hv$, a  convex nc rational function $f$ 
   in $\hv$ variables  and an analytic nc rational mapping  
  $q:M(\C)^\gv  \dashrightarrow M(\C)^\hv$ 
  such that $r=f\circ q$.
\end{theorem}

 The realization of equation \eqref{e:ratstate} is {\bf minimal} if \index{minimal realization}
\[
 \spann \{w(\AB_1\JK,\dots,\AB_\gv\JK,\AB_1^*\JK,\dots,\AB_\gv^*\JK)c : w\in\Langle x,\tilde{x}\Rangle\} = \C^d,
\]
 where $\langle x,\tilde{x}\rangle$ is the free monoid on the $2\gv$ freely noncommuting \index{$\tilde{x}$}
 variables $(x_1,\dots,x_\gv,\tilde{x}_1,\dots,\tilde{x_\gv}).$
An nc rational function regular at $0$ admits a minimal 
realization, which is readily computable and unique up to {\it  similarity} and
 in the symmetric case unique up to {\it unitary} similarity; 
see \cite[Section 4]{HMV06} or \cite[Section 6]{Vol18}, Remark \ref{r:apology} and Subsection \ref{sec:realizations}.

 Given a tuple $E\in M_d(\C)^\gv$, let 
  $\ran E$ denote the span of the \index{$\ran E$}
 ranges of the $E_j.$ 
  We can now state our second main result.

\begin{theorem}
 \label{t:isplush-intro}
   Assuming the realization of equation \eqref{e:ratstate} is minimal,
  $r$ is plush in a neighborhood of $0$ if and only if $P\JK P$ and $P_*\JK P_*$
  are both positive semidefinite, where $P$ and $P_*$ are the projections
 onto $\ran \AB$ and $\ran \AB^*$ respectively.
\end{theorem}

\begin{rem}\rm
  Since minimal realizations
 for  nc rational functions are 
 efficiently computable, Theorem \ref{t:isplush-intro}
 implies that so is determining whether
  an  nc rational function is plush. \qed
\end{rem}

There is one further result that merits inclusion in this introduction.
In \cite{HMV06} and  \cite{PT-Droyal} (see also \cite{PT-D17})
nc rational functions that are convex in a neighborhood of $0$
are characterized in terms of {\it butterfly representations}. 
 Below is an alternate characterization in the spirit of Theorem 
 \ref{t:isplush-intro}.

\begin{theorem}
 \label{t:intro-cvx}
 Assuming the realization of equation \eqref{e:ratstate} is minimal,
  $r$ is convex in a neighborhood of $0$ if and only if $Q\JK Q$ is 
 positive semidefinite, 
  where $Q$ is the projection onto $\ran \AB +\ran \AB^*.$
\end{theorem}

\subsection{Background and motivation}
 Given  $\varphi$, a  perhaps  matrix-valued symmetric nc 
  rational function, let $\mathfrak{P}_{\varphi}(n) =\{X\in \mat{n}(\C)^\gv: \varphi(X)\succ 0\}.$
 Let $\mathfrak{P}_\varphi$ denote the sequence $\left (\mathfrak{P}_\varphi(n)\right)_n.$
 In the case $\varphi$ is a polynomial,  $\mathfrak{P}_\varphi$ 
 is the free  analog of a basic semialgebraic set. In several complex
 variables, Levi pseudoconvex sets are described in terms of
 plurisubharmonic functions. Pushing this  analogy, if $\varphi$ is
 plush, we say $\mathfrak{P}_\varphi$ is a \df{free  pseudoconvex set}.
 Free pseudoconvex sets are natural for the free analog of several complex variables,
 particularly as domains for uniform polynomial approximation \cite{AM14,AM15} (see also \cite{BMV,AHKM18}).
 However, our primary motivation
 for studying nc plush functions and free  pseudoconvex sets arises
 in another way.\looseness=-1

 Given a tuple $\AB\in M_r(\C)^\gv$ and $X\in \mat{n}(\C)^\gv,$ let
\[
 L_\AB(X) = I_r\otimes I_n -\sum \AB_j\otimes X_j - \sum \AB_j^* \otimes X_j^*
\]
 and let
\[
 \mathfrak{P}_{\AB}(n) =\{X\in \mat{n}(\C)^\gv: L_\AB(X) \succ 0\}.
\]
 It is evident that each $\mathfrak{P}_\AB(n)$ is a convex subset of $\mat{n}(\C)^\gv$.
 The set $\mathfrak{P}_\AB(1)\subset \C^\gv$ is a \df{spectrahedron}.
 Thus spectrahedra form  a class of  convex subsets more general than polytopes, but 
 yet with a type of finitary representation.
Spectrahedra appear in several branches of mathematics,
such as convex optimization and real algebraic geometry \cite{BPR13}.
They also play a key role
in the solution of the Kadison-Singer paving conjecture \cite{MSS15}, and
the solution of the Lax conjecture \cite{HV07}.
It is natural
 to call the sequence $\mathfrak{P}_\AB = (\mathfrak{P}_\AB(n))_n$ a
 \df{free  spectrahedron}.   
Free spectrahedra arise naturally in applications such as systems engineering \cite{dOHMP09} 
and control theory \cite{HKMS}. They are also
intimately connected to the theories of matrix convex
sets, operator algebras and operator systems and completely positive maps \cite{EW,HKMjems,Pau,PSS}.

  By the main result of \cite{HMjoca} and also
 \cite{HMannals}, each $\mathfrak{P}_\varphi(n)$ is convex if and only if 
 $\mathfrak{P}_\varphi$ is a free  spectrahedron; that is,
  there exists a  $d$ and tuple $\AB\in M_d(\C)^\gv$ such that
  $\mathfrak{P}_\varphi = \mathfrak{P}_\AB$. In particular,  a basic free 
 semialgebraic set is convex if and only if it is a free  spectrahedron.

 Motivated by systems engineering considerations \cite{Skelton}, a problem is to determine,
 given a free  semialgebraic set $\mathfrak{P}_\varphi$ that is not 
 necessarily convex, if there is a free  spectrahedron
 $\mathfrak{P}_\AB$ and an analytic nc  rational mapping 
  $q:\mathfrak{P}_\varphi \to \mathfrak{P}_\AB$
 that is proper, or better still bianalytic.  Informally, the problem is
 to achieve convexity via change of variables.   Note that, in any case,
 the matrix-valued rational function $\psi = L_\AB\circ q$ is plush and 
 if $q$  is bianalytic, then $\mathfrak{P}_\varphi = \mathfrak{P}_{\psi}.$
 On the other hand, if $\varphi$ is plush, then by Theorem \ref{t:main-intro}  there exists
 a convex function $f$ in $\hv$ variables  and an analytic rational 
 mapping $q: M(\C)^\gv\dashrightarrow  M(\C)^\hv$  such that
 $\varphi = f\circ q.$ Now the set $\mathfrak{P}_f\subset M(\C)^h$ is convex
 and hence, by \cite{HMjoca}, there exists $A \in M_d(\C)^\gv$ 
 such that $\mathfrak{P}_f = \mathfrak{P}_A.$ Further, 
  $q:\mathfrak{P}_\varphi  \to \mathfrak{P}_A$ is proper. 
  Summarizing,  there is a proper analytic rational change of variables from $\mathfrak{P}_\varphi$
 to a convex set if and only if there is a plush  rational
 function $\psi$ such that   $\mathfrak{P}_\varphi =\mathfrak{P}_\psi.$

 Of course, in the case there exist distinct
  bianalytic rational mappings $q:\mathfrak{P}_\varphi \to\mathfrak{P}_B,$ 
 and $s:\mathfrak{P}_\varphi\to \mathfrak{P}_E$, 
 then there is a non-trivial bianalytic
 rational mapping $t:\mathfrak{P}_\AB \to \mathfrak{P}_{E}$. The articles
 \cite{AHKM18,HKMV} classify, up to some mild hypotheses, the
 triples $(\mathfrak{P}_\AB, \mathfrak{P}_E,t)$ where $t:\mathfrak{P}_\AB\to \mathfrak{P}_E$
 is an nc rational bianalytic mapping. 
Automorphisms of free domains such as balls
 have been considered by a number of authors including 
 \cite{MT16,MS,Po2,SSS18}. 

\subsection{Readers' guide} 
 Beyond this introduction, the  paper is organized as follows.
 Formulas for various derivatives of a symmetric descriptor realization, 
 a discussion of minimal realizations, a  canonical decomposition of
 the complex Hessian and a  preliminary version of Theorem
 \ref{t:isplush-intro} are collected in the next section,  Section \ref{sec:derivative}.  
 Theorem \ref{t:isplush-intro} is proved in
 Section \ref{sec:necessary}.  Theorem \ref{t:intro-cvx} 
 is proved in Section \ref{sec:compose} and 
 the half of Theorem \ref{t:main-intro} 
 that says the composition
 of a convex rational function and an analytic rational function
 is plush is obtained as a corollary. The proof 
 of Theorem \ref{t:main-intro} is completed in Section \ref{sec:main-result}.
 We conclude this introduction with the following remark.\looseness=-1

\begin{rem}\label{r:apology}
Throughout the text we will refer to several existing realization
 theoretic structural theorems, 
for example on convex polynomials,  rational functions, etc., 
that are scattered across the literature. However, in this paper we consider 
functions in variables $x$ and $x^*$, while in the existing literature, most 
statements involve symmetric or hermitian variables, or variables $x$ and $x^T$ 
evaluated on real matrices. The reason these results can be applied 
 in the present setting  has two justifications. 
 Firstly, for each of the required statements, the version 
for symmetric variables (and symmetric matrix functions) and the version 
for hermitian variables (and hermitian matrix functions) have essentially the
same proofs; in some cases, e.g. \cite{Vol18}, this was outlined explicitly. 
Secondly, to each function $f$ in $\gv$ variables $x_1,\dots,x_\gv$ and their 
adjoints $x_1^*,\dots,x_\gv^*$ one can associate a function $s$ in $2\gv$ 
hermitian variables $y_1,\dots,y_{2\gv}$ via
\[
 \begin{aligned}
s(y_1,\dots,y_{2\gv})
=f(y_1+iy_{\gv+1},\dots,y_\gv+iy_{2\gv},y_1-iy_{\gv+1},\dots,y_\gv-iy_{2\gv}),\\
f(x_1,\dots,x_\gv,x_1^*,\dots,x_\gv^*)
=s\left(
\frac{x_1+x_1^*}{2},\dots,\frac{x_\gv+x_\gv^*}{2},\frac{x_1-x_1^*}{2i},\dots,\frac{x_\gv-x_\gv^*}{2i}
\right).
\end{aligned}
\]

These transforms then enable us to freely move between the $(x,x^*)$-setting 
and the hermitian setting from the preceding papers.
\qed
\end{rem}

\section{Plush preliminaries}
 \label{sec:derivative}
 Let $r$ denote a symmetric descriptor realization as in equation \eqref{e:ratstate}.
 As preliminary results and background, this section contains  formulas for the derivative, complex Hessian  and (full) 
 Hessian of $r;$   a precisely stated preliminary version of Theorem \ref{t:main-intro}; and 
 a discussion of minimal descriptor realizations.

\subsection{Derivatives and the Hessians}
 Given $r$ as in equation \eqref{e:ratstate}, let 
\begin{equation}
\label{e:Delta}
\Del(x)=( \JK -  \Lam_\AB(x)  -  \Lam_{\AB^*}(x^*) )^{-1},
\end{equation}
 and given $X\in \mat{n}(\C)^\gv$ and assuming the inverse exists,
\begin{equation}
\label{e:DeltaX}
 \Del(X)=(\JK\otimes I_n-\Lam_\AB(X)-\Lam_{\AB^*}(X^*))^{-1}.
\end{equation}
 Thus $r(x)= c^* \Del(x) c$ and $r(X,X^*) = (c\otimes I_n)^* \Del(X) (c\otimes I_n).$
Straightforward direct calculation shows that the derivative 
 $r_x$ with respect to $x,$ the complex Hessian $r_{x,x^*}$ and 
 the full Hessian  $r^{\prime\prime}$  of $r$ are given by
 \begin{equation}
\begin{aligned}
r_x(x,x^*)[h,h^*]  & = c^* \
\Del(x) \  \Lam_\AB(h)  \ \Del(x) \
c \\
\label{eq:snd}
 r_{x,x^*}(x,x^*)[h,h^*]  & = 
c^* \ \Del(x)\  \Lam_{\AB}(h)^*  \ \Del(x)\  \Lam_\AB(h)  \ \Del(x)\ c
\\ & \phantom{=\,} + c^* \
\Del(x) \  \Lam_\AB(h)  \ \Del(x)
\  \Lam_{\AB}(h)^*  \ \Del(x) \
c,
\end{aligned}
\end{equation}
and 
\begin{equation}
\label{eq:trd}
\begin{split}
r^{\prime\prime}(x,&x^*)[h,h^*]  = r_{x,x}[h,h^*]  +2r_{x,x^*}[h,h^*]
+ r_{x^*,x^*}[h,h^*] \\
&= 2 \Big[c^* \Del(x)  \Lam_{\AB}(h) \Del(x) \Lam_\AB(h) \Del(x) c 
+ c^* \Del(x)  \Lam_{\AB}(h)^* \Del(x) \Lam_\AB(h) \Del(x) c \\
&\phantom{=\ } + c^* \Del(x)  \Lam_\AB(h) \Del(x) \Lam_{\AB}(h)^* \Del(x) c
+ c^* \Del(x)  \Lam_{\AB}(h)^* \Del(x) \Lam_\AB(h)^* \Del(x) c \Big]\\
&=2 c^* \Del(x) \left(\Lam_{\AB}(h)+\Lam_{\AB}(h)^*\right) 
    \Del(x) \left(\Lam_{\AB}(h)+\Lam_{\AB}(h)^*\right) \Del(x) c,
\end{split}
\end{equation}
respectively.

\subsection{Decomposing the complex Hessian}
A subset $\Omega \subset M(\C)^\gv$ is a sequence $\Omega = (\Omega(n))_n$, where
 $\Omega(n)\subset \mat{n}(\C)^\gv$.  The set $\Omega$ is \df{closed with respect to direct sums}
 if $X\in \Omega(n)$ and $Y\in \Omega(m)$ implies,
\[
 X\oplus Y =\begin{bmatrix} X & 0\\0&Y \end{bmatrix}
   = \left ( \begin{bmatrix} X_1 & 0\\ 0 & Y_1 \end{bmatrix}, \dots,
   \begin{bmatrix} X_\gv & 0\\0&Y_\gv \end{bmatrix} \right ) \in \Omega(n+m).
\]
 The descriptor realization $r$ as in \eqref{e:ratstate} is \df{plush on $\Omega$} if, 
 $r_{x,x^*}(X,X^*)[H,H^*]\succeq 0$ for each $n$,
 each $X\in \Omega(n)$ and each $H\in \mat{n}(\C)^\gv.$
 Given $X,\tX,H\in \mat{n}(\C)^\gv$, let
\begin{equation*}
\begin{split}
r_\downarrow(X,\tX)[H]
   &=C^* \Delta(X)\Lambda_B(H)^*\Delta(\tX)\Lambda_B(H)\Delta(X)C,\\
r_\uparrow(X,\tX)[H]
  &=C^* \Delta(\tX)\Lambda_B(H)\Delta(X)\Lambda_B(H)^*\Delta(\tX)C,
\end{split}
\end{equation*}
where $C=c\otimes I_n.$

\begin{prop}
\label{p:realizationplushimplies}
   Suppose $\Omega\subset M(\C)^\gv$ is closed with respect to direct sums.
 Then the nc rational function $r$ as in \eqref{e:ratstate} is 
  plush on $\Omega$   if and only if
\begin{equation}
\label{e:rupdown}
\begin{split}
r_\downarrow(X,\tX)[H]\succeq 0
\quad \text{and}\quad
r_\uparrow(X,\tX)[H]\succeq 0
\end{split}
\end{equation}
for all $X,\tX\in\Omega$ and $H\in M(\C)^\gv.$  
\end{prop}

\begin{proof}
 Given $X,\tX\in\Omega(n)$ and $H\in \mat{n}(\C)^\gv,$ define
\[
\widehat{X}= \bem X & 0  \\
0 & \tX
\eem
 \qquad
\widehat{H}=\bem 0 & 0 \\
H & 0
\eem.
\]
 Since $\Omega$ is closed with respect to direct sums, $\widehat{X}\in \Omega(2n).$
For notational convenience, let $\Del =\Del(X),$ 
 $\tDel=\Del(\tX)$ and $C=c\otimes I_n$ and observe
\begin{multline*}
 r_{x,x^*}(\widehat{X},\widehat{X}^*)[\widehat{H},\widehat{H}^*]= \\
 \bem  C^*  &0\\  0& C^*  \eem
  \
 \bem  \Del &0\\  0& \tDel  \eem
 \bem
  0 &  \Lam_{\AB}(H)^* \\ 0    & 0  \eem
 \ \bem
 \Del &0\\
 0& \tDel
 \eem
\bem
 0 & 0 \\\  \Lam_\AB(H) & 0  \eem
\ \bem
 \Del &0\\
 0& \tDel
 \eem
  \bem
 C  &0\\
 0& C
 \eem\\
 +    \bem
 C^*  &0\\
 0& C^*
 \eem \
\bem
 \Del &0\\
 0& \tDel
 \eem\
\bem
 0 & 0 \\ \  \Lam_\AB(H)   &0
 \eem
\bem
 \Del &0\\
 0& \tDel
 \eem
\
\bem
0 &\  \Lam_{\AB}(H)^* \\0  &0
 \eem\
 \bem
 \Del &0\\
 0& \tDel
 \eem\ \
 \bem
 C  &0\\
 0& C
 \eem
\end{multline*}
and thus
\begin{equation*}
0\preceq r_{x,x^*}(\widehat{X},\widehat{X}^*)[\widehat{H},\widehat{H}^*]=
\bem r_\downarrow(X,\tX)[H] & 0 \\ 0 & r_\uparrow(X,\tX)[H]\eem,
\end{equation*}
an identity from which the result immediately
follows. 
\end{proof}

Let
$\Pdown,\Pup:\C^d\to \C^d$ denote the projections onto $\rng \AB$ and $\rng \AB^*$
respectively.

\begin{cor}
 \label{cor:plushimplies}
   If $\Omega\subset M(\C)^\gv$  is closed with respect to direct sums and 
   both $\Pdown \Delta(X) \Pdown$ and $\Pup \Delta(X)\Pup$ are positive 
   semidefinite for each tuple $X\in \Omega$, then $r$ is plush on $\Omega$. 
\end{cor}

\begin{proof}
 For $X\in \Omega(n)$ and $H\in \mat{n}(\C)^\gv$, 
 since the range of $\Lambda_B(H) \Delta(X)$ lies in $\rng \AB\otimes \C^n$, 
  the result follows from Proposition \ref{p:realizationplushimplies}
  by choosing $\tX=X$ and using either of the inequalities of 
 \eqref{e:rupdown}.
\end{proof}

 For $\kv$ a positive integer and
 $\ve>0$ the {\bf (column) free  ball} $\bbB_\ve\subset M(\C)^\kv$ \index{free ball}
  of radius $\ve$ 
  is the sequence $\bbB_\ve = (\bbB_\ve(n))_n$ given by 
\[
  \bbB_\ve(n) =\left\{X\in\mat{n}(\C)^\kv\colon \sum_{j=1}^\kv X_j^* X_j\prec \ve^2 I_n \right\}
  \subset \mat{n}(\C)^\kv.
\] 
 Evidently free balls are closed with respect to direct sums.
 An nc rational mapping $q:M(\C)^\kv\dashrightarrow M(\C)^\hv$ regular at $0$ takes the form
 $q=\begin{bmatrix} q_1 & q_2 &\dots & q_\hv \end{bmatrix},$ where
 each $q_j\in \freef$ is regular at $0.$  Let
 $q_x(x)[h] = \begin{bmatrix} (q_1)_x(x)[h] & \dots & (q_\hv)_x(x)[h]\end{bmatrix}.$
 Thus $q_x(X)[H]\in \mat{n}(\C)^\hv$ for $X,H\in \mat{n}(\C)^\gv.$

\begin{cor}
\label{cor:compose-}
If $r$ is a symmetric nc rational function in $\hv$ variables
that is plush on some free ball and  if
 $q:M(\C)^\gv\dashrightarrow M(\C)^\hv$  is an nc rational mapping
that is regular at $0$ with $q(0)=0,$ then $\varphi= r\circ q$ is plush 
on some free ball.
\end{cor}

\begin{proof}
We assume $r$ is  given in equation \eqref{e:ratstate} and is plush on $\bbB_\ve\subset M(\C)^\hv.$
By \eqref{eq:snd} and the chain rule
for $X,H\in \mat{n}(\C)^\hv,$ %
\[
\begin{split}
\varphi_{x,x^*}&(X,X^*)[H,H^*] 
\\  & =  C^* \Delta(q(X)) \Lambda_\AB(q_x(X)[H])^* \Delta(q(X)) \Lambda_\AB(q_x(X)[H]) \Delta(q(X))C
\\&\phantom{=\,} +C^* \Delta(q(X)) \Lambda_\AB(q_x(X)[H]) \Delta(q(X))\Lambda_\AB(q_x(X)[H])^* \Delta(q(X)) C 
\\ & = C^* \Delta(Y) \Lambda_\AB(E)^* \Delta(Y) \Lambda_\AB(E) \Delta(Y)C
+C^* \Delta(Y) \Lambda_\AB(E) \Delta(Y)\Lambda_\AB(E)^* \Delta(Y) C 
\\ & = r_\downarrow(Y,Y)[E] + r_\uparrow(Y,Y)[E],
\end{split}
\]
 where $Y=q(X)$ and $E=q_x(X)[H]$ and $C=c\otimes I_n.$  
 Since $q(0)=0$, there is a $\delta>0$ such that for each $n$ and each  $X\in \bbB_\delta(n) \subset \mat{n}(\C)^\gv$,  we have $Y=q(X) \in \bbB_\ve(n)\subset \mat{n}(\C)^\hv.$ 
By Proposition \ref{p:realizationplushimplies}, $r_\downarrow(Y,Y)[E],r_\uparrow(Y,Y)[E]\succeq 0$
and hence $\varphi_{x,x^*}(X,X^*)[H,H^*]\succeq 0$ for all $n$,  $X\in \bbB_\delta(n)$ and 
 $H\in \mat{n}(\C)^\gv.$ Thus $\varphi$ is plush on $\bbB_\delta.$
\end{proof}

\subsection{Rational functions and realizations}
\label{sec:realizations}
  A foundational result in  the theory of nc rational functions 
  is the fact that a realization is  minimal if and only if it is observable 
 and controllable \cite{BGM}. For symmetric nc rational functions, the corresponding realizations 
  are symmetric descriptor realizations.\looseness=-1

  The formal domain of a symmetric descriptor realization $r$ as in equation
 \eqref{e:ratstate}  is the set of those $X\in M(\C)^\gv$ 
 such that  $J-\Lam_\AB(X) - \Lam_\AB(X)^*$ is invertible.  Two symmetric
 descriptor realizations $r$ and $s$ are \df{equivalent} if $r(X,X^*)=s(X,X^*)$
 for all $X$ in the intersection of the domains of $r$ and $s.$ 
 A \df{symmetric nc rational function} is an equivalence class
 of symmetric descriptor realizations. A symmetric descriptor realization is {\bf minimal} \index{minimal realization}
  if its
 size ($d$ in the case of the  realization of
 equation \eqref{e:ratstate}) is minimum amongst all elements
 of its equivalence class.

\begin{prop}
A symmetric descriptor realization
\begin{equation}\label{eq:miniff}
c^*(\JK-\sum_j\AB_jx_j-\sum \AB_j^*x_j^*)^{-1}c
\end{equation}
of size $d$ is minimal if and only if
\[
 \C^d = \spann\left\{w(\AB_1\JK,\dots,\AB_\gv\JK,\AB_1^*\JK,\dots,\AB_\gv^*\JK)c: w\in\Langle x,\tilde{x}\Rangle \right\}.
\]
\end{prop}

\begin{proof}
Since $K^{-1}=K,$ the realization \eqref{eq:miniff} can be rewritten as a monic realization
\begin{equation}\label{eq:miniff1}
c^*\JK(I-\sum_j\AB_j\JK x_j-\sum \AB_j^*\JK x_j^*)^{-1}c.
\end{equation}
By \cite[Theorem 9.1]{BGM}, the realization \eqref{eq:miniff1} is minimal if and only if
\[\begin{split}
\C^d &= \spann\left\{w(\AB_1\JK,\dots,\AB_\gv\JK,\AB_1^*\JK,\dots,\AB_\gv^*\JK)c: w\in\Langle x,\tilde{x}\Rangle \right\},\\
\C^{1\times d} &= \spann\left\{c^*\JK w(\AB_1\JK,\dots,\AB_\gv\JK,\AB_1^*\JK,\dots,\AB_\gv^*\JK): w\in\Langle x,\tilde{x}\Rangle \right\}.
\end{split}\]
However, since $K^*=K,$ these two equalities are clearly equivalent.
\end{proof}

\section{A realization theoretic characterization of plush nc   rational functions}
\label{sec:necessary}
 This section is devoted to the proof of Theorem \ref{t:isplush-intro}, restated 
 as Theorem \ref{t:local} below.  A \df{free neighborhood of $0$} in $M(\C)^{\gv}$ is
 a sequence $\Omega = (\Omega(n))_n$, where $\Omega(n)\subset \mat{n}(\C)^\gv$
 is open and that contains some free ball.  In particular, a free ball is
 a free neighborhood of $0.$

  Throughout this section, 
 $r$ is a symmetric descriptor realization (of size $d$)
 as in  \eqref{e:ratstate} and $P$ and $P_*$
 are the orthogonal projections onto $\ran \AB$ and $\ran \AB^*$ respectively.

\begin{theorem}
\label{t:local}
 If  $\Pup \JK \Pup$ and $\Pdown \JK \Pdown$ are positive semidefinite, then $r$ is plush 
 on a free  ball; that is, there is an $\ve>0$ such that 
 $r_{x,x^*}(X,X^*)[H,H^*]\succeq 0$ for all $n$,  $X\in\bbB_\ve(n)$ and $H\in \mat{n}(\C).$

Conversely, if $r$ is plush on a free  ball and
the realization \eqref{e:ratstate} is minimal,
then $\Pup \JK \Pup$ and $\Pdown \JK\Pdown$ are both positive semidefinite.
\end{theorem}

Theorem \ref{t:local} follows by combining Propositions \ref{p:left} and 
\ref{prop:right} below. Recall the notations  $\Del(x)$ and $\Del(X)$ from equations \eqref{e:Delta} and
 \eqref{e:DeltaX}.

\begin{lemma}\label{l:indep}
If the realization \eqref{e:ratstate} is minimal, then 
for  every $\ve>0$ there exists an $n$,
an $X\in \bbB_\ve(n)$  %
and a vector  $v\in \C^n$ such that 
\[
  z=\Delta(X)(c\otimes v)\in\C^d\otimes\C^n
\]
has $d$ linearly independent components in $\C^n$; that is,
writing $z=\sum_{j=1}^d e_j \otimes z_j$, the set $\{z_1,\dots,z_d\}\subset \C^n$
 is linearly independent.
\end{lemma}

\begin{proof}
 Substitute $x_j=y_j+iy'_j$ to obtain the matrix-valued symmetric nc rational function
 $\widetilde{\Delta}(y,y^\prime) = \Delta(x)$ in  $2\gv$ symmetric variables and apply a
 hermitian version of \cite[Lemmas 7.2 and 7.4]{HMV06} (which hold because the local-global
principle of linear dependence also works in hermitian settings, cf.~\cite{BK})
 to obtain the desired conclusion.
\end{proof}

\begin{lemma}\label{l:span}
Let $\{e_1,\dots,e_d\}$ denote a basis for $\C^d$ and $\kv$ be a positive integer.
If  $z=\sum_{i=1}^d e_i\otimes z_i\in \CC^d\otimes \CC^n$  and
$\{z_1,\dots,z_d\}$ is a linearly independent set of vectors in $\C^n$, then for any
 $E\in M_d(\C)^\kv,$
\[
\left\{\Lambda_E(H)z\colon H\in \mat{n}(\C)^\kv\right\}
= \rng E  \otimes \C^n.
\]
\end{lemma}

\begin{proof}
We have
$$\Lambda_E(H)z= \sum_{j=1}^\kv (E_j\otimes  H_j)z= \sum_{j=1}^\kv\sum_{i=1}^d E_je_i\otimes H_jz_i.$$
Fix $1\le i_0\le d$, $1\le j_0\le \kv$ and an $f\in \C^n.$  Let $H_j=0$ for $j\neq j_0$ and
 let $H_{j_0}$ be such that $H_{j_0}z_i=0$ for $i\neq i_0$ and $H_{j_0}z_{i_0}= f.$  Then
$$\Lambda_E(H)z=E_{j_0}e_{i_0}\otimes f.$$
Since  $\cS =\left\{\Lambda_E(H)z\colon H\in \mat{n}(\C)^\kv\right\}\subset \C^d\otimes\C^n$
is a subspace, it follows that $\cS \supseteq [\rng E_{j_0}] \otimes \C^n$ and finally
that $\cS \supseteq [\rng E] \otimes \C^n.$ Since the reverse inclusion is
evident, the proof is complete.
\end{proof}

\begin{prop}[Necessity]
\label{p:left}
 Suppose $r$ as in \eqref{e:ratstate} is a minimal realization.  
 If there is an $\ve>0$ such that  
 $r_\downarrow(X,0)[H]\succeq0$ for all $n$,  all  $X\in\bbB_\ve(n),$
 and all $H\in \mat{n}(\C)^\gv,$  then $\Pdown \JK\Pdown\succeq0$. 
 In particular, if
$r$ is plush on some free  ball, then $\Pdown \JK \Pdown$
and  $\Pup \JK \Pup$ are both positive semidefinite.
\end{prop}

\begin{proof}
Since the realization \eqref{e:ratstate} is assumed minimal, 
 Lemma \ref{l:indep} implies there
exists an $n$, a tuple $X\in \bbB_\ve(n),$  and a vector $v$ 
such that 
$z=\Delta(X)(c\otimes I)v\in\C^d\otimes\C^n$ has $d$ linearly 
independent components in $\C^n$. By assumption, for this $X$ and $v$ and all $H$,
\begin{equation}\label{eq:psd}
  v^* r_\downarrow(X,0)[H] v =  z^*\Lambda_B(H)^*(\JK\otimes I)\Lambda_B(H)z \ge0.
\end{equation}
By Lemma \ref{l:span}, 
  $\{\Lambda_B(H)z:H\in \mat{n}(\C)^\gv\}= [\ran \AB ]\otimes \C^n.$ Thus 
$\Pdown \JK\Pdown\succeq0$ by \eqref{eq:psd}.
\end{proof}

\begin{prop}[Sufficiency]
\label{prop:right}
 Let $Q$ and $R$ denote the inclusions of  $\ran \AB$ and $\ran \AB^*$ into $\C^d$ 
 respectively. If $P\JK P$ and $P_*\JK P_*$ are both positive semidefinite, then 
 there is an $\ve>0$ such that for each $n$ and $X\in \bbB_\ve(n)$, 
 both  $(Q\otimes I_n)^* \Delta(X) (Q\otimes I_n)$
 and  $(R\otimes I_n)^*\Delta(X)(R\otimes I_n)$ are positive semidefinite
 and $r$ is plush on $\bbB_\ve.$
\end{prop}

Proposition \ref{prop:right} can be deduced as a consequence of the construction
in Section \ref{sec:main-result}. A direct proof follows and starts with some geometric definitions.  

For the  $d\times d$ signature matrix $\JK$, 
a subspace $\cN\subset \C^d$ is \df{$\JK$-nonnegative} if\looseness=-1
\[
 \langle \JK h,h\rangle \ge 0
\]
for all $h\in \cN$.  Note that the hypotheses $P\JK P$ is positive semidefinite  
in Proposition \ref{prop:right}
is equivalent to $Q^*\JK Q\succeq 0$ and to the condition that the range of $\AB$ is $\JK$-nonnegative.
If $\cN$ is $\JK$-nonnegative, then  $[h,g]=\langle \JK h,g\rangle$ defines 
a semi-inner product on $\cN$. In particular,  if $h\in \cN$ and $[h,h]=0$, then
$[h,f]=0$ for all $f\in \cN$ and hence
\[
 \cN^0 =\{h\in \cN: \langle \JK h,h\rangle =0\} \subset \cN
\]
is a subspace, called the  \df{$\JK$-neutral subspace} of $\cN$.  
Now suppose $\cN^+\subset \cN$ is a complementary subspace to $\cN^0$; that
 is $\cN^+\cap \cN^0 =\{0\}$ and $\cN = \cN^+ + \cN^0$. If 
 $h\in \cN^+$ and $h\ne 0$, then
\[
  \langle \JK h,h  \rangle >0. 
\]
 Because $\cN^+$ is finite dimensional, it follows that there is an $\eta>0$ such that
\[
 \langle \JK h,h\rangle \ge \eta \|h\|^2,
\]
for $h\in \cN^+$.  Thus, letting $V:\cN^+\to \C^d$ denote the inclusion, we have 
 $V^*\JK V \succeq \eta I_{\cN^+} >0$.\looseness=-1

\begin{proof}[Proof of Proposition \ref{prop:right}]
For notational purposes, let $\cR$ and $\cR_*$ denote 
$\ran \AB$ and $\ran \AB^*$ respectively. 
Let $\cR^0$ denote the $\JK$-neutral subspace of $\cR$. 
There is a $1\ge \eta>0$
and a subspace $\cR^+\subseteq \cR$ such that,
\begin{enumerate}[({\rm Q}1)]
\item  $\cR^0 + \cR^+ = \cR$ and $\cR^0\cap \cR^+=\{0\}$;
\item  $Q_+^* \JK Q_+ \succeq \eta I_{\cR^+}$, where $Q_+$ denotes the inclusion
of $\cR^+$ into $\C^d.$
\end{enumerate}
Likewise (after changing $1\ge \eta>0$ if needed), 
there exists a subspace $\cR_*^+\subseteq \cR_*$ such that
\begin{enumerate}[({\rm R}1)]
\item  $\cR_*^0 + \cR_*^+ = \cR_*$ and $\cR_*^0\cap \cR_*^+=\{0\}$;
\item   $R_+^* \JK R_+ \succeq \eta I_{\cR_*^+}$, where $R_+$ denotes the inclusion
of $\cR_*^+$ into $\C^d.$
\end{enumerate}

Let $\Phi(x,x^*) = \Lambda_{\AB}(x) + \Lambda_{\AB}(x)^*$. 
There is an $\ve>0$ such that if $X \in\bbB_\ve,$ 
then $\sum_{j=1}^\infty \|\Phi(X,X^*) \|^j <\frac{\eta}{2}$.   It suffices to prove, if
$X\in \bbB_\ve(n)$, then 
$(Q\otimes I_n)^*\Delta(X) (Q\otimes I_n)\succeq 0$ 
and $(R\otimes I_n)^* \Delta(X)  (R\otimes I_n)\succeq 0.$

Suppose $X\in \bbB_\ve(n)$ and thus 
 $\|\Phi(X,X^*)\|<\frac{\eta}{2}\le \frac 12$. In particular,
$I_{dn}-\Phi(X,X^*)(\JK\otimes I_n)$ is
invertible and 
\begin{equation}
\label{e:Del-v-Phi}
 \Delta(X) = (K\otimes I_n -\Phi(X,X^*))^{-1} = [K\otimes I_n] (I-\Phi(X,X^*)[K\otimes I_n])^{-1}.
\end{equation}

 Note, if $\gamma\in \cR^0$ and $\delta\in \C^d$, then $\AB_j\delta \in \cR$
and hence  $\delta^* \AB_j^* \JK \gamma=0$. Thus $\AB_j^* \JK\gamma=0$ and hence,
 for $z\in \C^n,$
\[
\begin{split}
 \Phi(X,X^*) (\JK\otimes I_n)\, (\gamma \otimes z)
   &=   \Phi(X,X^*) (\JK\gamma  \otimes z)\\
   &=  \sum \AB_j\JK\gamma \otimes X_jz  + \sum \AB_j^* \JK\gamma \otimes X_j^*z\\
   &= \sum \AB_j\JK\gamma \otimes X_jz  \in \cR\otimes \C^n.
\end{split}
\]
It follows that
\[
 [I_{dn}-\Phi(X,X^*)(\JK\otimes I_n)]( \gamma\otimes z) \in \cR \otimes \C^n.
\]
Hence
\begin{equation*}
\cS_X:= [I_{dn}-\Phi(X,X^*)(\JK\otimes I_n)] \cR^0\otimes \C^n  \subseteq  \cR \otimes \C^n.
\end{equation*}
Since $I_{dn}-\Phi(X,X^*) (\JK\otimes I_n)$ is invertible, $\dim \cS_X = n\, \dim \cR^0.$
Furthermore, using equation \eqref{e:Del-v-Phi} and $(Q\otimes I_n)\cS_X =\cS_X,$
\[
\begin{split}
(Q^*&\otimes I_n) \Delta(X) (Q\otimes I_n)\cS_X \\
&= (Q^*\otimes I_n) \Delta(X) [I_{dn}-\Phi(X,X^*)(\JK\otimes I_n)]\cR^0\otimes \C^n \\
 &=   (Q^*\JK\otimes I_n) \cR^0 \otimes \C^n = [Q^* \JK\cR^0] \otimes \C^n =\{0\},
\end{split}
\]
 since $\cR^0$ is the $K$-neutral subspace of $\cR.$ Thus
\begin{equation*}
  \cS_X %
  \subseteq  \ker (Q^*\otimes I)\, \Delta(X)  \, (Q\otimes I) \subseteq \cR\otimes \C^n.
\end{equation*}

Since $\|\Phi(X,X^*)\|<\frac{\eta}{2}$ and $(Q_+\otimes I_n)^* (\JK\otimes I_n) (Q_+\otimes I_n) \succeq \eta I_{\cR^+}$,
\begin{equation}
\label{e:arrow4}
\begin{split}
( Q_+^* \otimes I_n)\, \Delta(X) \, (Q_+\otimes I_n) 
 & =( Q_+^* \otimes I_n)\, (\JK -\Phi(X,X^*))^{-1} \, (Q_+\otimes I_n) 
\\ & \succeq Q_+^*\JK Q_+\otimes I_n  - \sum_{k=1}^\infty  \|\Phi(X,X^*)\|^k I_{\cR^+}
\succeq \frac{\eta}{2} I_{\cR^+}.
\end{split}
\end{equation}
In particular, $\cS_X \cap (\cR^+\otimes \C^n) =\{0\}.$

Summarizing,
\begin{enumerate}
\item $\cS_X,\, \cR^+\otimes \C^n \subseteq \cR\otimes \C^n$;
\item $(Q^*\otimes I_n)\, \Delta(X)  (Q\otimes I_n) \cS_X = \{0\};$
\item $\dim \cS_X = n\, \dim \cR^0$ and $\dim \cR^+\otimes\C^n=n\, (\dim \cR -\dim \cR^0);$
\item  $(Q_+^*\otimes I_n)\Delta(X) (Q_+\otimes I_n)\succeq \frac{\eta}{2} I_{\cR^+}$ %
(see equation \eqref{e:arrow4});
\item  $\cS_X \cap (\cR^+\otimes \C^n) =\{0\}.$
\end{enumerate}
It follows that $\cS_X\, + [\cR^+\otimes \C^n] = \cR$
 and if $\delta \in \cS_X$
and $\gamma\in \cR^+\otimes \C^n$, then
\[
\begin{split}
 \langle (Q^*\otimes I_n)\Delta(X) & (Q\otimes I_n) (\delta+\gamma),\delta +\gamma \rangle
\\  &  = \langle (Q^*\otimes I_n)\Delta(X) (Q\otimes I_n) \gamma,\gamma\rangle
 \ge \frac{\eta}{2} \|\gamma\|^2\ge  0.
\end{split}
\]
Hence $(Q^*\otimes I_n)\Delta(X) (Q\otimes I_n)\succeq 0$ as desired. By symmetry,
$(R^*\otimes I_n) \Delta(X)  (R\otimes I_n)\succeq 0.$ Thus $Q^*\Delta(x)Q$
and $R^*\Delta(x)R$ are both positive semidefinite in a neighborhood of $0$. 
Thus $r$ is plush by Corollary \ref{cor:plushimplies}.
\end{proof}

\section{Convex nc rational functions}
\label{sec:compose}
 The main result of this section is Theorem \ref{t:intro-cvx}, restated 
 and proved as Proposition \ref{prop:cvx} below.   
 An immediate consequence is the fact that if a symmetric nc rational function
 is convex in a free ball, then it is plush in a free ball.  Thus, combined with
 Corollary \ref{cor:compose-}, Theorem \ref{t:intro-cvx} establishes one-half  of Theorem 
 \ref{t:main-intro}.

 Throughout this section, $f$ denotes the symmetric descriptor realization,
\begin{equation}
\label{e:realizecvx}
 f(x) = v^* (J-\Lambda_A(x)-\Lambda_A(x)^*)^{-1} v,
\end{equation}
 where $\hv$ is a positive integer,  $A\in M_d(\C)^\hv$ and $0\neq v\in\C^d.$  

\begin{prop}\label{prop:cvx}
If  $\rng A + \rng A^*$ is a $J$-nonnegative
 subspace of $\C^d,$ then $f$ is convex in a neighborhood of $0$. 

Conversely, if the realization \eqref{e:realizecvx} is minimal and $f$ is convex in a neighborhood
of $0$, then $\rng A+\rng A^*$ is a $J$-nonnegative subspace of $\C^d.$  
\end{prop}

\begin{cor}
\label{cor:convex-plush}
 If $f$ is convex, then $f$ is plush.
\end{cor}

\begin{proof}
 By Proposition \ref{prop:cvx} both $\rng A$ and $\rng A^*$ are $J$-nonnegative  subspaces.
 An application of Theorem \ref{t:local} completes the proof.
\end{proof}

\begin{cor}
\label{cor:compose}
Suppose $f$ is a symmetric nc rational function
 in $\hv$ variables,  
and $q:M(\C)^\gv\dashrightarrow M(\C)^\hv$ is an analytic nc rational mapping.
If $f$ is convex in a neighborhood of $0$, then $r=f\circ q$ is plush
 in a neighborhood of $0.$
\end{cor}

\begin{proof}
 By Corollary \ref{cor:convex-plush}, since $f$ is convex it is plush. 
 The result now follows from Corollary \ref{cor:compose-}.
\end{proof}

The proof of Proposition \ref{prop:cvx} uses Lemma \ref{l:precvx} below.

\begin{lemma}
\label{l:precvx}
 Let $\cJ\in M_d(\C)$ be a signature matrix. If $\cN \subset \C^d$ is a $\cJ$-nonnegative
 subspace, then there is a $\delta>0$ such that if $n$ is a positive integer,
 $T\in M_d(\C)\otimes M_n(\C)$ is selfadjoint, $\ran T \subset \cN \otimes \C^n$  
 and $\|T\|< \delta$, then 
\[
 (P\otimes I_n) \, (\cJ\otimes I_n - T)^{-1}\, (P\otimes I_n) \succeq 0,
\]
where $P$ is the projection onto $\cN.$
\end{lemma}

\begin{proof}
 Let $\cN_0$ denote the $\cJ$-neutral subspace of $\cN$. 
 In particular, $P \cJ w_0 = 0$ for $w_0\in \cN_0$.
 Let $\cN_+$ denote the  orthogonal complement of $\cN_0$ in $\cN$. 
 Hence $\cN_0\oplus \cN_+ =\cN$ and $\cN_+$ is a $\cJ$-strictly positive subspace.
 In particular, there is an $\eta>0$ such that
 if $w\in \cN_+$, then $\langle \cJ w,w\rangle \ge \eta \langle w,w\rangle.$
 Choose $\delta =\frac{\eta}{1+\eta}<1$ and note 
 $\sum_{j=1}^\infty \delta^j  =\eta.$

 Now let $n$ be given.
 Let $\tcJ= \cJ\otimes I_n$ and note that $\tcN :=\cN\otimes \C^n$ is 
 $\tcJ$-nonnegative and $\tcN_0:=\cN_0\otimes \C^n$ is its $\tcJ$-neutral subspace.
 Since $\cP =P\otimes I_n$ is the projection onto the $\tcJ$-nonnegative subspace $\tcN$
 and $\tcN_0$ is neutral,
  $\cP  \tcJ w_0=0$ for $w_0 \in \cN_0\otimes \C^n$.  Moreover, if 
 $w\in \tcN_+$, then $\langle \tcJ w,w\rangle \ge \eta \langle w,w\rangle.$

 Fix $T$ as in the statement of the lemma.   Since $\delta<1$, 
 $\tcJ-T$ is invertible with the inverse given by the convergent series\looseness=-1
\[
  (\tcJ  -T)^{-1} = \tcJ + \tcJ\sum_{j=1}^\infty (T \tcJ)^j.
\]
If $w_0\in \tcN_0$ and $w_+\in \tcN_+$, then, since $\langle \tcJ w_0,v\rangle =0=\langle w_0,\tcJ v\rangle$
 for $v\in \tcN$ and since  $\rng T \subset \tcN$,
\[
 \langle w_0, \tcJ(T\tcJ)^j w_+ \rangle =0 =\langle w_0, \tcJ (T\tcJ)^j w_0\rangle,
\]
for all nonnegative integers $j.$
Hence
\[
\begin{split}
\langle (\tcJ-T)^{-1}(w_0+w_+),w_0+w_+\rangle  
 & =  \langle \tcJ w_+,w_+\rangle  + \langle \sum_{j=1}^\infty \tcJ(T\tcJ)^j w_+,w_+\rangle \\
& \ge  (\eta  - \sum_{j=1}^\infty \|T\|^j)\|w_+\|^2 \\
 & \ge  (\eta - \eta)\|w_+\|^2 =0
\end{split}
\]
and the conclusion of the lemma follows.
\end{proof}

\begin{proof}[Proof of Proposition~\ref{prop:cvx}]
 Let $\Phi(x) = \Lambda_A(x)+\Lambda_A(x)^*$ and let 
\[
 \Gamma(x) = \left(J-\Lambda_A(x)-\Lambda_A(x)^*\right)^{-1},
\]
and for $X\in \mat{n}(\C)^\hv$ for which the inverse exists,
\[
 \Gamma(X) = \left(J\otimes I_n-\Lambda_A(X)-\Lambda_A(X)^*\right)^{-1}.
\]

By \eqref{eq:trd},
\[
 f^{\prime\prime}(x,x^*)[h,h^*] = 
 2 v^* \Gamma(x) \Phi(h) \Gamma(x) \Phi(h) \Gamma(x) v.
\]
Moreover, $f$ is convex in a neighborhood of $0$ if and only if there is a
$\eta>0$ such that for all $n$, all  $X\in \bbB_\eta(n)$ and all $H\in \mat{n}(\C)^g$,
\[
 f^{\prime\prime}(X,X^*)[H,H^*]
 = 2 (v^*\otimes I_n) \Gamma(X) \Phi(H) \, \Gamma(X) \, \Phi(H)\Gamma(X)(v\otimes I_n) \succeq 0,
\]
by \cite[Proposition 5.1]{HMV06} and Remark \ref{r:apology}.

Now suppose $\cN =\rng A+\rng A^*$ is $J$-nonnegative. By Lemma \ref{l:precvx},
 there is a $\delta>0$ such that for each $n$ and each tuple $X\in \bbB_\delta(n),$ 
\[
\begin{split}
(P\otimes I_n) \,&  \Gamma(X) \, (P\otimes I_n) \\
 =&  (P^*\otimes I_n) \, (J\otimes I_n-\Lambda_A(X) -\Lambda_A(X)^*)^{-1} (P\otimes I_n) \succeq 0,
\end{split}
\]
where $P$ is the projection onto $\cN$. Since $\Phi(H)$ maps into the range
of $P\otimes I_n$, it follows that $f^{\prime\prime}(X,X^*)[H,H^*]$ is positive semidefinite
 for $X\in \bbB_\delta.$  Thus $f$ is convex on $\bbB_\delta.$

Conversely, suppose there is an $\ve>0$ such that $f$ is convex
on $\bbB_\ve \subset M(\C)^\hv.$  Without loss of generality we may
assume the realization of equation \eqref{e:realizecvx} is minimal. 

For $H,\wtH\in M(\C)^\hv$ let
$$\Psi(H,\wtH)=\Lambda_A(H)+\Lambda_A(\wtH)^*.$$
 Given $X,\wtX,H,\wtH \in \mat{n}(\C)^\hv$,   let
\[
 \YY =\begin{bmatrix} X& 0\\0&\wtX \end{bmatrix}, \ \  
 \KK= \begin{bmatrix} 0&H\\\wtH &0\end{bmatrix},
\]
let 
\[
\begin{split}
 f_\downarrow(X,\wtX)[H,\wtH ]& =  \Gamma(X) \Psi(H,\wtH)  \Gamma(\wtX)\Psi(H,\wtH)^*  \Gamma(X), \\
 f_\uparrow(X,\wtX)[H, \wtH ] & =  \Gamma(\wtX)\Psi(H,\wtH)^* \Gamma(X)\Psi(H,\wtH)  \Gamma(\wtX),
\end{split}
\]
and observe 
\[
 \Phi(\KK) = \begin{bmatrix} 0 & \Psi(H,\wtH) \\ \psi(H,\wtH)^* & 0 \end{bmatrix}
\]
 and therefore
\[
\begin{split}
 f^{\prime\prime}(\YY,\YY^*)[\KK,\KK^*] = 
 2(v\otimes I_{2n})^* & \begin{bmatrix} f_\downarrow(X,\wtX)[H,\wtH] & 0\\
 0& f_\uparrow(X,\wtX)[H,\wtH] \end{bmatrix}
 (v\otimes I_{2n}).
\end{split}
\]
Hence, since  $f^{\prime\prime}(\YY,\YY^*)[\KK,\KK^*]$  is positive semidefinite
for $\YY\in \bbB_\ve(2n)$ and $\KK\in M_{2n}(\C)^\hv$, 
\[
\begin{split}
(v\otimes I_n)^* \, f_\downarrow(X,\wtX)[H,\wtH] \, (v\otimes I_n)&\succeq0,\\ \
(v\otimes I_n)^* \, f_\uparrow(X,\wtX)[H,\wtH] \,  (v\otimes I_n)
& \succeq 0,
\end{split}
\]
for all $X,\wtX\in \bbB_\ve(n)$  and  $H,\wtH\in \mat{n}(\C)^\hv.$  In particular,
 for each $X\in \bbB_\ve(n)$   and  $H,\wtH\in \mat{n}(\C)^\gv,$
\[\begin{split}
  0 & \preceq 
   (v\otimes I_n)^* f_\downarrow(X,0)[H,\wtH] (v\otimes I_n) \\
   & =  (v\otimes I_n)^* \Gamma(X) \Psi(H,\wtH) (J\otimes I_n)\Psi(H,\wtH)^* \Gamma(X) (v\otimes I_n).
\end{split}\]
Using minimality of the realization for $f$, 
by Lemmas \ref{l:indep} and \ref{l:span} there exist $X\in\bbB_\ve(n)$ and $u\in\C^n$ such that the set
$$\{\Psi(H,\wtH)\Gamma(X)(v\otimes u): H,\wtH\in\mat{n}(\C)^\hv \}$$
spans $(\rng A+\rng A^*)\otimes \C^n$.
Hence $PJP\succeq 0$, where $P$ is the projection onto $\rng A+\rng A^*$.
\end{proof}

\section{Plush rationals are composite of a convex with an analytic}
\label{sec:main-result}

\def\ltAplay{{\begin{bmatrix} D \\ 0  \\ \gamma^* \\ \gamma^* \end{bmatrix}}}
\def\rtAplay{{\begin{bmatrix} 0 & D_* & \gamma_* &  \gamma_* \end{bmatrix}}}

\def\rtA{{\phi_*}}
\def\ltA{{\phi}}
\def\tr{\widetilde r}

\def\ttq{{\tt q}}
\def\ttf{{\tt f}}

\def\basisw{{\bw}}

\def\ltAb{\psi}
\def\rtAb{\psi_*}

\def\bbA{{\mathbb A}}

In this section we prove Theorem \ref{t:main-intro}, restated as 
Theorem \ref{thm:main} below. It is the main result of this paper.

\begin{theorem}\label{thm:main}
Suppose $r$ is a symmetric nc rational function.
 If $r$ is plush in a neighborhood of the origin,
 then there exists a positive integer $\hv$, a convex nc rational function $f$
 in $\hv$ variables, 
 and an analytic nc rational mapping $q:M(\C)^\gv\dashrightarrow M(\C)^\hv$ such that $r=f\circ q$. 
 Moreover, a choice of $f$ and $q$ is explicitly constructed from
 a minimal realization of $r$. See formulas \eqref{eq:f} and
\eqref{eq:q} and Subsection \ref{sssec:dep}.
\end{theorem}

\subsection{A formal recipe for $f$ and $q$} 
\label{sec:formal}
We may assume $r$ is a minimal descriptor realization as in formula \eqref{e:ratstate}.
There exist nonnegative integers $\ap$ and $\bq$ such that
\[
 K=\begin{bmatrix} I_\ap & 0 \\ 0 & -I_\bq \end{bmatrix}.
\]
 Since $r$ is, by assumption, plush in a neighborhood of $0$,
both  $\rng B$ and $\rng B^*$ are  $K$-nonnegative
by Theorem \ref{t:local}. Hence we may assume $\ap\ge 1$ (as otherwise $r$ is constant).
 Likewise, we may assume $\bq\ge 1$ as otherwise $r$ is convex in a neighborhood of $0$,
 and therefore
 plush by Corollary \ref{cor:convex-plush}, and the conclusion of the theorem
 follows  upon choosing $q(x)=x.$ and $f=r.$

A subspace $\cP$ is a \df{maximal $K$-nonnegative subspace} if 
$\cP$ is $K$-nonnegative if $\cN$ is nonnegative 
 with $\cP\subset \cN$, then $\cN=\cP$. 
It is well known that, in this case, the dimension of $\cP$ is $a$
and moreover, there is a contraction $\rho:\C^\ap \to \C^\bq,$
 known as the \df{angular operator} for $\cP$ \cite{A79},  such that
 $\cP$ is the range of the map
\[
\begin{bmatrix} I_\ap \\ \rho \end{bmatrix} : \C^\ap \to \C^\ap\oplus \C^\bq.
\]

Let $\ao,\ao_*:\C^\ap \to \C^\bq$ denote the angular operators for 
maximal $K$-nonnegative 
subspaces $\cP$ and $\cP_*$ containing $\rng B$ and $\rng B^*$ respectively. 
 Let $P, P_*$  denote the orthogonal projections
  onto $\cP$ and $\cP_*$ respectively. 
Let  $\ax$ denote the set of words in $x_1,\dots,x_\gv$
and { $\axplus = \ax\setminus \{1\}$};
these are analytic words (no $x_j^*$s).

 If $Q$ is a positive semidefinite matrix, then, up to unitary equivalence, 
it is of the form $Q_+\oplus 0$, where $Q_+$ is positive definite. 
 Hence, again up to unitary equivalence,
 the Moore-Penrose pseudoinverse $Q^\dagger$ of $Q$ takes the form $Q_+^{-1}\oplus 0$. 
 In particular, the ranges of $Q$ and $Q^\dagger$ are the same. 
 Let $D$ and $D_*$ denote the positive (semidefinite) square roots of
 $I_{\ap}-\ao^*\ao$ and $I_{\ap}-\ao_*^*\ao_*,$ respectively.  Define
$\ltAb: %
\C^{\ap+\bq}  \to
 \C^\ap \oplus \C^\ap \oplus \C \oplus \C $
 and
$\rtAb:
 \C^\ap \oplus \C^\ap \oplus \C \oplus \C  \to \C^{\ap+\bq}$ by
\[
\ltAb:={{\begin{bmatrix} 
D^\dagger \bem I_\ap & \ao^*\eem 
 \\
0_{ \ap\times (\ap+\bq)} 
\\ 
c^*
 \\ 
c^*
 \end{bmatrix}}}
\qquad 
\text{and} 
\qquad
\rtAb:={{ \begin{bmatrix}
0_{(\ap+\bq)\times \ap} & \
\bem I_\ap \\\rho_* 
\eem D_*^\dagger \ \ \
c 
 & 
c
 \end{bmatrix}.
}}
\]
The definition of the formal 
representation $(J,\bbA,v)$ of $\ttf$
 is as follows. 
\begin{enumerate}
 \item \label{i:forma11}
Define, for each 
${w\in{\axplus}}$, 
\[
\bbA_w:= \ltAb \; w(KB) \; P_*K \rtAb 
= \ltAb \; w(KB) \; K \rtAb 
\in M_{2\ap+2}(\C).
\]
\item  Let 
\begin{equation}
 \label{d:formalJ}  
    J=I_\ap\oplus I_\ap\oplus1\oplus-1 \in M_{2\ap+2}(\C)
\end{equation} 
 and 
\begin{equation*}
v=\begin{bmatrix}0\\ 0 \\ s \\ t \end{bmatrix} \in 
\C^\ap\oplus \C^\ap\oplus \C\oplus\C =\C^{2\ap+2}.
\end{equation*}
Here we take
 $s,t\in \C$ such that $s-t=1$ and $s+t=c^*Kc$ 
 (hence $s^2-t^2 = c^* K c$).
\end{enumerate}
The expression  
\[
\begin{split}
 \ttf(y) & :=  v^* \left( J -\sum_{w\in{\axplus}} 
( \bbA_w y_w  + \bbA_w^* y_w^*) \right)^{-1}v \\
& := v^* \left( J -\sum_{n=1}^\infty \sum_{|w|=n}
( \bbA_w y_w  + \bbA_w^* y_w^*) \right)^{-1}v
\end{split}
\]
defines a formal power series in infinitely many variables $y_w,y_w^*$;
more precisely, it is an element of the completion of $\C\langle y_w,y_w^*:w\in\axplus\rangle$  with respect to the descending chain of ideals
\[J_n=\Big(y_{w_1}\cdots y_{w_\ell}\colon \sum_{k=1}^\ell |w_k|=n\Big).\]
In the spirit of Proposition \ref{prop:cvx} one could say that $\ttf$ is {\it formally convex.}
Let 
\[
{\ttq}(x_1,\dots,x_\gv)={\ttq}(x) = (w)_{w\in \axplus}.
\]
Thus $\ttq$ is an analytic polynomial mapping with infinitely many outputs.

\begin{theorem}
\label{t:formal}
  Viewing $y_w=q_w(x)=w(x)$ and composing $\ttf$ with $\ttq$ gives
\begin{equation*}
r(x) = \ttf (\ttq(x)),
\end{equation*}
in the ring of formal power series. 
\end{theorem}

Theorem \ref{t:formal} is proved in Subsection \ref{sec:tformalproof} and it is used
in the proof of Theorem \ref{thm:main} appearing in Subsection \ref{sec:precise}.
 Referring to the variables $y_w,y_w^*$ as \df{intermediate variables},
 $\ttf$ depends on infinitely many intermediate variables and $\ttq$, while 
 a function of the variables $x$, outputs the intermediate variables. 
 In Subsection \ref{sec:precise} as part of the proof of Theorem
 \ref{thm:main},  rational $f$ and $q$ are constructed using only
 finitely many intermediate variables.

\subsection{Proof of Theorem~\ref{t:formal}}
\label{sec:tformalproof}
Let $\SA$ denote the set of strictly alternating words in two letters $\{\bx,\by\}$.
Hence, $\bx \by\bx\by,\bx \by\bx\by\bx,$ and
$\by \bx \by\bx\by$ are
examples of such  words. 
We do not include
the empty word in $\SA$.  Using the fact that 
$\rtAb J \ltAb =0,$ and hence 
$\bbA_w J \bbA_u =0$  for $u,w\in \axplus$, compute
 \begin{equation}
\label{e:sumoveralt}
\begin{split}
\ttf(\ttq(x)) & =  
v^* \left (I - \Lambda_{J\bbA}(\ttq(x)) -\Lambda_{J\bbA^*}(\ttq(x)^*) \right )^{-1} Jv \\
& = v^* \left [\sum_{k=1}^\infty  \left (\Lambda_{J\bbA}(\ttq(x)) + \Lambda_{J\bbA^*}(\ttq(x)^*) \right )^n  \, \right ] Jv
   + v^*Jv \\
& =  v^* \left[\sum_{w\in \SA} 
 w(\Lambda_{J\bbA}(\ttq(x)),\Lambda_{J\bbA^*}(\ttq(x)^*)) \right ] Jv
   + v^*Jv.
\end{split}
\end{equation}
The next and longest part of the argument simplifies
$w(\Lambda_{JA}(q(x)),\Lambda_{JA^*}(q(x)^*))$ for $w\in \SA$.

\begin{lemma}
 \label{l:crazyPB}  For $1\le j,\ell \le \gv$,
\begin{equation*}
\begin{split}
   B_\ell^* K\, [\psi^* J \psi]\,  KB_j &=  B_\ell^* KB_j \\
   B_\ell K\, [\psi^* J \psi]\, KB_j^* &=   B_\ell KB_j^*.
\end{split}
\end{equation*}
\end{lemma}

The proof of Lemma \ref{l:crazyPB} uses the following construction.
First note that the projection $P$ onto $\cP$ is given by 
\[
 P = \begin{bmatrix} I \\ \ao \end{bmatrix} (I+\ao^*\ao)^{-1}
   \begin{bmatrix} I & \ao^* \end{bmatrix}
\]
 and a similar formula holds for $P_*,$  the projection onto $\cP_*.$ Set
\[
 \EE_j = (I+\ao^* \ao)^{-1}    \begin{bmatrix} I & \ao^* \end{bmatrix} \,
        B_j  \, \begin{bmatrix} I \\ \ao_* \end{bmatrix} (I+\ao_*^* \ao_*)^{-1}
  \in M_\ap(\C).
\]
Thus,
\[
 P B_j P_* = \begin{bmatrix} I \\ \ao \end{bmatrix} \, \EE_j \,
   \begin{bmatrix} I & \ao_*^* \end{bmatrix}.
\]
Finally, since $(\ker B_j)^\perp = \rng B_j^* \subseteq \cP_*$, it follows that
 $PB_jP_*= B_j$. Hence,
\begin{equation}
 \label{eq:BE}
     B_j =  \begin{bmatrix} I \\ \ao \end{bmatrix} \, \EE_j \,
   \begin{bmatrix} I & \ao_*^* \end{bmatrix}.
\end{equation}

\begin{proof}[Proof of Lemma~\ref{l:crazyPB}]
 Compute,
\[
 \psi^* J\psi = \begin{bmatrix} I \\ \rho \end{bmatrix} \,
   (D^\dagger)^2 \begin{bmatrix} I & \rho^* \end{bmatrix}.
\]
Thus, using formula \eqref{eq:BE}, 
$(I-\rho^*\rho) (D^\dagger)^2 (I-\rho^*\rho) = D^2 (D^\dagger)^2 D^2=I-\rho^*\rho$
 and 
\[
 \begin{bmatrix} I_\ap & \rho^* \end{bmatrix} \, K \, \begin{bmatrix} I_\ap \\ \rho \end{bmatrix} = I-\rho^* \rho,
\]
 it follows that
\[
\begin{split}
 B_\ell^*    K\psi^* J \psi KB_j 
    = & \begin{bmatrix} I \\ \rho_* \end{bmatrix} E_\ell^*  (I-\rho^*\rho) 
          (D^\dagger)^2 (I-\rho^*\rho) E_j \begin{bmatrix} I & \rho_*^*\end{bmatrix} \\
  = &  \begin{bmatrix} I \\ \rho_* \end{bmatrix} E_\ell^*  (I-\rho^*\rho) 
          E_j \begin{bmatrix} I & \rho_*^*\end{bmatrix} \\
   = & \, B_\ell^* KB_j.
\end{split}
\]
 The other identity can be proved in a similar fashion. We omit the details.
\end{proof}

For notational purposes, let
$\tOm$ and $\tGm$ denote  the formal power series
\begin{equation*}
\tOm(x) = \sum_{n=1}^\infty  (K\Lambda_B(x))^n
   =  \sum_{w\in\axplus} w(KB) \, w(x) = \sum_{j=1}^\gv KB_j \sum_{w\in \ax} w(KB) \, x_j w(x)
\end{equation*}
and
\begin{equation*}
   \tGm(x^*) = \sum_{n=1}^\infty (K (\Lambda_B(x))^*)^n
  = \sum_{w\in \axplus} K w(KB)^*K\,  w(x)^* = K \tOm(x)^* K.
\end{equation*}
With these notations,
\begin{equation}
 \begin{split}
 \label{eq:lamJA}
 \Lambda_{J \bbA}(\ttq(x)) & =  \sum_{w\in{\axplus}}  J \bbA_w \ttq_w(x)
 =
 \sum_w J  \ltAb \, w(KB) P_* K   \, \rtAb  \;  w(x)
\\ & =
J \ltAb \big[ \sum_w    \, w(KB)  w(x) \big] \; P_* K    \, \rtAb 
 =   J \ltAb \; \tOm(x) \;   P_* K   \rtAb \\
  &=   J \ltAb \; \tOm(x) \;   K   \rtAb  
 \end{split}
 \end{equation}
and 
\begin{equation}
 \begin{split}
 \label{eq:lamJAs}
 \Lambda_{J \bbA^*}(\ttq(x)^*) 
 & =  
 \sum_{w\in{\axplus}}  J \bbA_w^* \ttq_w(x)^*
 = \sum_w J  \rtAb^*   K P_* \, w(KB)^*   \, \ltAb^*  \;  w(x)^*\\
&  =   
 J \rtAb^* K P_* \ \tOm(x)^* \;    \ltAb^* \\
 &=   
 J \rtAb^* K \ \tOm(x)^* \;    \ltAb^* = J \rtAb^*\; \tGm(x^*) K \ltAb^*.
 \end{split}
 \end{equation}
Further, using Lemma \ref{l:crazyPB},
\begin{equation}
\label{e:crazyPBapplied}
 \tGm(x^*) K\,  [\psi^* J \psi] \tOm(x) 
   = \tGm(x^*)\, \tOm(x).
\end{equation}
Combining equations \eqref{eq:lamJA}, \eqref{eq:lamJAs} and \eqref{e:crazyPBapplied} gives
   \def\Des2{{\Delta_*^2}}
\[
\begin{split}
 \Lambda_{J \bbA^*}(\ttq(x)^*) 
 \Lambda_{J \bbA}(\ttq(x)) 
 &=    J \ltAb_*^*\; [\tGm(x^*) K \psi^* J \ltAb \; \tOm(x)] \;   K   \rtAb  \\
 &=   J \ltAb_*^*\; [\tGm(x^*)  \tOm(x)] \;   K   \rtAb.
\end{split}
\]
Similarly,
\begin{equation*}
 \tOm(x) K \ltAb J \ltAb^* \tGm(x^*) = \tOm(x)  \tGm(x^*).
\end{equation*}
Thus,
\[
\begin{split}
 \Lambda_{J \bbA}(\ttq(x))  \Lambda_{J \bbA^*}(\ttq(x)^*) 
 &=  J \ltAb \; [\tOm(x)  \rtAb J \rtAb^*   \tGm(x^*)]\;  K     \ltAb^* \\
  &=J \ltAb \; [\tOm(x)\ \tGm(x^*)]\;  K     \ltAb^*.
\end{split}
\]

Next turn to an  alternating word, say  $w(\bx,\by)=\bx \by \cdots \bx \by$
where $\bx$ and $\by$ each appear $N$ times.
Writing  $\tOm$ and $\tGm$ instead of $\tOm(x)$ and $\tGm(x^*)$
and computing as above,  
\begin{equation}
\label{e:xyxy}
\begin{split}
 w(\Lambda_{J\bbA}(\ttq(x)),   \Lambda_{J\bbA^*}(\ttq(x)^*))
&=   J \ltAb \; \tOm \;\tGm \; \tOm \cdots  \tGm \; \tOm \; \tGm \; K \ltAb^* \\
 &=  J\psi \; w(\tOm,\tGm)\; K\psi^*.
\end{split}
\end{equation}

The last step in the proof of Theorem \ref{t:formal}  is to match moments as follows.
 The right hand side of equation \eqref{e:xyxy} is the sum over all terms of the  form
\[
T= J \ltAb \;  (K\Lambda_B(x))^{n_1} (K\Lambda_B(x)^*)^{m_1} \cdots 
   (K\Lambda_B(x))^{n_N} (K\Lambda_B(x)^*)^{m_N} \; K\ltAb^*,
\]
 for positive integers $n_j,m_j$.  
Further, 
\[
 v^* T Jv = c^* (K\Lambda_B(x))^{n_1} (K\Lambda_B(x)^*)^{m_1} \cdots 
   (K\Lambda_B(x))^{n_N} (K\Lambda_B(x)^*)^{m_N} \; Kc
\]
and 
\[
 v^* J v= s^2-t^2 = c^*Kc.
\]
Hence, letting $\mathcal T$ denote all possible products of the form $T$ (save
for the empty product) and $\axyplus$ the nonempty  words in $(x,x^*)$, 
\begin{equation}
\label{e:(J,A)v(K,B)}
\begin{split}
v^* [ \sum_{w\in \SA}  &
 w(\Lambda_{J\bbA}(\ttq(x)),\Lambda_{J\bbA^*}(\ttq(x)^*)) ]  Jv
   + v^*Jv
 =   \sum_{T\in\mathcal T} v^*TJv + v^*Jv 
\\ &=  c^*\, [ \sum_{u\in \axyplus}  u(KB,KB^*) ]\, Kc  + c^*Kc 
 \\ & = c^*(K-\Lambda_B(x)-\Lambda_B(x)^*)^{-1} c,
\end{split}
\end{equation}
since  the sum over $w\in \SA$ gives all possible products of 
$K \Lambda_B(x), K \Lambda_B(x)^*$ save for the empty product ($I$). 
Combining equations \eqref{e:(J,A)v(K,B)} and \eqref{e:sumoveralt}
completes the proof of Theorem \ref{t:formal}.

\subsection{Proof of Theorem~\ref{thm:main}}
\label{sec:precise}
In this section  Theorem \ref{thm:main} is deduced from Theorem \ref{t:formal}.
It is possible to prove Theorem \ref{thm:main} directly.

\subsubsection{A recipe for $f$ and $q$ having finitely many intermediate variables}
  In the construction of $\ttf$ and $\ttq$ in Subsection \ref{sec:formal}  the intermediate space 
 has infinitely many variables. In this subsection that construction is 
 refined, under the additional assumption that $\{B_1,\dots,B_\gv\}$
 is linearly independent,  to produce rational  convex $f$ and analytic $q$ 
 having an intermediate space with finitely 
 many variables that are shown, in Subsection \ref{sec:endofproof}, 
 to satisfy the conclusion of Theorem \ref{thm:main}. Finally, Subsection \ref{sssec:dep} 
 shows how to pass from linear dependence to independence of the set $\{B_1,\dots,B_\gv\}.$

To construct $f$ and $q$, 
let $\{C_1,\dots,C_\hv\}$ denote a basis for the algebra generated by $\{KB_1,\dots,KB_\gv\}$
and, without loss of generality, assume $C_j=KB_j$ for $1\le j\le \gv$
(since we are assuming $\{B_1,\dots,B_\gv\}$ is linearly independent)
and for $\gv+1\le j \le \hv,$  that
\[
C_j=\basisw_j(KB)  
\]
for some (non-empty) word $\basisw_j.$ Note that $\hv\le (\ap+\bq)^2$ as $KB_j\in M_{\ap+\bq}(\C).$
In particular,  we can set $\basisw_j = x_j$ for $1\le j\le \gv$.

There is an $\hv$-tuple $\Xi\in M_\hv(\C)^\hv$ 
such that  for each $1\le j,k\le \hv$,
\begin{equation}
\label{e:definesXi}
 C_j C_k =\sum_{s=1}^\hv  (\Xi_k)_{j,s} C_s,
\end{equation}
though we will be mostly interested in $1\le j,k\le \gv$. 
Moreover, for $1\le j\le \hv$ and a word $w$ in $(x_1,\dots,x_\hv)$,
\begin{equation}
 \label{e:Xicontonic}
 C_j \, w(C) = \sum_{s=1}^\hv w(\Xi)_{j,s} C_s, 
\end{equation}
by \cite[Lemma 2.5]{HKMV}.

Define $f$ and $q$ as follows.
\ben
\item
Let $J$ denote the symmetry matrix from equation \eqref{d:formalJ} and,
for $1 \leq s \leq \hv,$ define
\begin{equation*}
  A_s:= \bbA_{\basisw_s} = \ltAb \; \basisw_s(KB) \; P_*K \rtAb 
=\ltAb \; \basisw_s(KB) \; K \rtAb.
\end{equation*}
Set
\begin{equation}\label{eq:f}
 f(y) = v^* (J-\Lambda_A(y) -\Lambda_A(y)^*)^{-1} v,
\end{equation}
where $A=(A_1, \dots A_\hv)\in M_{2a+2}(\C)^\hv$ and
$y=(y_1,\dots,y_\hv)$. 

Since
\[
\rng A + \rng A^* \subseteq \{ \begin{bmatrix} D^\dagger \begin{bmatrix} I_\ap &  \rho^*\end{bmatrix}  u  \\  
    D^\dagger_* \begin{bmatrix} I_\ap & \rho_*^* \end{bmatrix}  v  \\ 
     c^* (u+z) \\ c^* (u+z)  \end{bmatrix}: u,z  \in \C^\ap \oplus \C^\bq \},
\]
it follows that  $\rng A+\rng A^*$ is $J$-nonnegative and therefore
$f$ is convex, by Proposition \ref{prop:cvx}.

\item
Let  
$b(y)= \begin{bmatrix} b_1(y) & \dots & b_\hv(y) \end{bmatrix}$ 
denote the 
map associated to $\Xi$ by
\begin{equation}
\label{eq:contonMap}
 b(y) = y (I-\Lambda_\Xi(y))^{-1} .
\end{equation}
For $1\le s \le \hv$, let
\begin{equation}\label{eq:q}
q_s(x)= b_s(x_1,\dots,x_\gv,0,\dots,0) = \sum_{j=1}^\gv \sum_{w\in {\ax}} (w(\Xi))_{j,s} x_j  w.
\end{equation}
\een
Evidently $q=\begin{bmatrix} q_1 & \cdots & q_\hv \end{bmatrix}$ 
is analytic and rational.

\begin{rem}
The nc rational mapping $b(y)$ of  \eqref{eq:contonMap}, 
associated to a tuple $\Xi$ satisfying
  \eqref{e:definesXi},  is a  \df{convexotonic} map,
see \cite[Section 1.1 and Lemma 2.5]{HKMV}.
Up to linear change of variables and an irreducibility
assumption, convexotonic maps are the only bianalytic
 maps between free spectrahedra \cite{AHKM18,HKMV}.
\end{rem}

\subsubsection{Proof that $r=f\circ q$}
\label{sec:endofproof}
Since $r=\ttf\circ\ttq$ by Theorem \ref{t:formal}, both $f$ and $q$ are rational, $f$ is convex  and $q$
 is analytic,  Theorem \ref{thm:main} 
in the case that $\{B_1,\dots,B_\gv\}$
 is linearly independent
 is a consequence
 of Proposition \ref{p:doesntmatter}.

\begin{prop}
\label{p:doesntmatter}
 $\ttf(\ttq(x))=f(q(x))$.
\end{prop}

\begin{proof}
Since
\[
 \ttf(\ttq(x)) = v^*\left(I-\Lambda_{J\bbA}(\ttq(x))-\Lambda_{J\bbA^*}(\ttq(x)^*)\right )^{-1} Jv
\]
 and
\[
 f(q(x)) = v^* \left(I-\Lambda_{JA}(q(x))-\Lambda_{JA^*}(q(x)^*)\right )^{-1} Jv,
\]
the conclusion  follows from Lemma \ref{l:contonicmagic} below.
\end{proof}

\begin{lemma}
\label{l:contonicmagic} 
 With notations as above,
\begin{equation*}
 \Lambda_{A}(q(x))  = \Lambda_{\bbA}(\ttq(x)).
\end{equation*}
\end{lemma}

 Recall the notation $C_j = \basisw_j(KB)$ for $1\le j\le \gv$ and that $C_j = KB_j$ for $1\le j\le \gv$.
 Thus, by equation \eqref{e:Xicontonic}, for $1\le j \le \gv$ and $w\in \ax$, 
\begin{equation}
\label{e:CjCa} 
    KB_j  \; w(KB)  = \sum_{s=1}^\hv \;    w(\Xi)_{j,s}  \; \basisw_s(KB) 
  = \sum_{s=1}^\hv \, w(\Xi)_{j,s} C_s.
\end{equation}

\begin{proof}
 Using the identity in  equation \eqref{e:CjCa} in the fourth  equality 
 \begin{equation*} 
   \begin{split}
\Lambda_{\bbA}(\ttq(x)) 
= & \ \ltAb \ \Lambda_{\big (w(KB)\big )_{\axplus} }(\ttq(x) ) \; P_* K\rtAb
\\ = & \ \ltAb \  \sum_{w \in {\axplus}}  w(KB)    \ \ttq_w(x)  \ P_*K \rtAb 
\\ = & \  \ltAb \  \sum_{u=1}^\gv \sum_{w \in {\ax}}  [ KB_u w(KB)] \;  x_u w(x)  \ P_*K \rtAb  
\\ = & \ 
   \ltAb \   \sum_{u=1}^\gv \sum_{w \in {\ax}} 
 [\sum_{j=1}^\hv \;    (w(\Xi))_{u,j} C_j  \; ] \;x_u  w(x)  \ P_*K \rtAb 
\\  = & \ 
 \sum_{j=1}^\hv  \ltAb \; C_j  \; P_*K \rtAb   \ [ \sum_{u=1}^\gv \sum_{w\in {\ax}}
 (w(\Xi))_{u,j} x_u \, w(x) ] 
 \\   = & \  \sum_{j=1}^\hv  \ltAb \; [ \basisw_j(KB) \; P_*K \rtAb ] \;  q_j(x) \\
 = &  \ \sum_{j=1}^\hv  A_j  q_j(x)  = \Lambda_A(q(x)). \qedhere
  \end{split}
 \end{equation*}
\end{proof}

\subsubsection{Linearly dependent $B_j$}\label{sssec:dep}
 To complete the proof of Theorem \ref{t:main-intro}, suppose, without loss of
 generality, that $1\le\kv\le \gv$ and $\{\Bp_1,\dots,\Bp_\kv\}$ is a basis for the span
 of $\{B_1,\dots,B_\gv\}$.  Let 
\[
 \rp(y) = c \left (K-\sum_{j=1}^\kv \Bp_j y_j  - \sum \Bp_j^* y_j^* \right )^{-1} c.
\]
Thus $\rp$ is a symmetric descriptor realization. There is a $\gv\times\kv$ matrix
 $M$ such that $r(x) = \rp(Mx)$. Moreover, since $\ran \Bp =\ran B$
 and $\ran \Bp^* = \ran B^*$ and since $r$ is assumed plush, Theorem \ref{t:local}
 implies $\rp$ is also plush.  Thus, by what has already been proved,  there exists a positive
 integer $\hv,$ an analytic nc rational mapping
 $\qp:M(\C)^{\kv} \dashrightarrow M(\C)^\hv$  and a convex nc rational function $f$ (in $\hv$ variables)
 such that $\rp(y) = (f\circ \qp)(y)$. Set $q(x)=\qp(Mx).$ Thus $q$ is an analytic
 nc rational mapping and $r=f\circ q.$

\end{document}